\documentclass{birkjour}
\usepackage{amsmath,amsfonts,amsthm,amssymb,verbatim,enumerate,quotes,graphicx,mathtools}
\usepackage{color}

\newcommand{\comments}[1]{}
\newtheorem{theorem}{Theorem}
\newtheorem*{theoremA}{Theorem A}
\newtheorem*{theoremB}{Theorem B}
\newtheorem*{theoremC}{Theorem C}

\def\C{\mathbb{C}}
\def\N{\mathbb{N}}
\def\Z{\mathbb{Z}}

\newcommand{\tef}{transcendental entire function}
\newcommand{\tmf}{transcendental meromorphic function}

\newcommand\qfor{\quad\text{for }}

\dedicatory{To Phil Rippon on the occasion of his 65th birthday}

\begin{document}
%
%
%
%
\title[On permutable meromorphic functions]{On permutable meromorphic functions}
\author{J. W. Osborne, \, D. J. Sixsmith}
\address{Department of Mathematics and Statistics \\
	 The Open University \\
   Walton Hall\\
   Milton Keynes MK7 6AA\\
   UK}
\email{john.osborne@open.ac.uk, david.sixsmith@open.ac.uk}
\thanks{The second author was supported by Engineering and Physical Sciences Research Council grant EP/J022160/1.}
%
%
%
%
\begin{abstract}
We study the class $\mathcal{M}$ of functions meromorphic outside a countable closed set of essential singularities. We show that if a function in $\mathcal{M}$, with at least one essential singularity, permutes with a non-constant rational map $g$, then $g$ is a M\"{o}bius map that is not conjugate to an irrational rotation.  For a given function $ f \in\mathcal{M}$ which is not a M\"{o}bius map, we show that the set of functions in $\mathcal{M}$ that permute with $ f $ is countably infinite. Finally, we show that there exist \tmf s $f: \C \to \C$ such that, among functions meromorphic in the plane, $f$~permutes only with itself and with the identity map.
\end{abstract}
\maketitle
%
%
%
%
\section{Introduction}
If $ f $ and $ g $ are meromorphic functions, then, in general, $f \circ g$ is not meromorphic.  In view of this, we let $\mathcal{M}$ be the class of functions $f$ with the following property; there is a countable closed set $S(f) \subset \widehat{\C} := \C \cup \{\infty\}$ such that $f$ is meromorphic in $\widehat{\C}\setminus S(f)$, and $S(f)$ is the set of essential singularities of $f$. All functions that are meromorphic in $ \C $ lie in $\mathcal{M}$, which is closed under composition. The dynamics of functions in the class $\mathcal{M}$ was considered by Bolsch \cite{MR1423240,MR1703869}.

In this short paper we extend some results of Baker and Iyer on permutable entire functions to permutable functions in the class $\mathcal{M}$. Here, if $f$ and $g$ are functions defined on a subset of $\widehat{\C}$, then we say that these functions are \emph{permutable}, or that $ f $ \emph{permutes} with $ g $, if 
\begin{equation}
\label{eq:perm}
f \circ g = g \circ f.
\end{equation}

We take (\ref{eq:perm}) to mean that, for all $z \in \widehat{\C}$, either $f(g(z)) = g(f(z))$, or else both $f(g(z))$ and $ g(f(z))$ are undefined.

The case where $f$ is transcendental entire and $g$ is a polynomial was considered by Baker \cite{MR0097532} and independently by Iyer \cite{MR0105574}. They proved the following.
\begin{theoremA}
\label{theo:BI}
Suppose that $g$ is a non-constant polynomial. Then there exists a {\tef} $f$ that permutes with $g$ if and only if $g$ is an affine map of the form 
\begin{equation}
\label{defg}
g(z) = z e^{2\pi i m/n} + c, \qfor \text{some } m, n \in \Z, \ c \in \C.
\end{equation}
\end{theoremA}

Throughout this paper, when we refer to a \tmf, we mean a function meromorphic in $ \C $ and with an essential singularity at infinity.  It is straightforward to use results in \cite{MR0265595} to generalise Theorem A to \tmf s. As this result is not stated in \cite{MR0265595}, we give it here. 
\begin{theorem}
\label{theo:perm}
Suppose that $g$ is a non-constant rational map. Then there exists a \tmf\ $f$ that permutes with $g$ if and only if $g$ is an affine map of the form (\ref{defg}).
\end{theorem}

Our main generalisation of Theorem A is to functions in the class $\mathcal{M}$. It is natural to state these results in terms of conjugacies. For suppose that $f$ and $g$ are permuting elements of $\mathcal{M}$, and that $g$ is \emph{conjugate} to a map $G\in\mathcal{M}$; in other words, there is a M\"{o}bius map $L$ such that $G = L^{-1} \circ g \circ L$. Then $F$ and $G$ are permuting elements of $\mathcal{M}$, where $F := L^{-1} \circ f \circ L$. 

It is well-known that a M\"{o}bius map $M$ is conjugate either to the map $z \mapsto z + 1$ (in which case $M$ is called \emph{parabolic}), or to the map $z \mapsto kz$, for some $k \in \C\setminus\{0\}$. In the second case, if $k = e^{i \theta}$, for $\theta$ rational (resp. irrational), then we say that $M$ is conjugate to a \emph{rational rotation} (resp. \emph{irrational rotation}).
\begin{theorem}
\label{theo:permK}
Suppose that $g$ is a non-constant rational map. Then there exists a function $f \in\mathcal{M}$, with $S(f) \ne \emptyset$, and that permutes with $g$, if and only if $g$ is a M\"{o}bius map that is not conjugate to an irrational rotation.
\end{theorem}

Baker \cite{MR0147650, MR0226009} also proved the following result.
\begin{theoremB}
\label{theo:Bak}
If an entire function $ f $ is not a polynomial of degree less than $ 2 $, then the set of all entire functions that permute with $ f $ is countably infinite.
\end{theoremB}
We show that a result of Bergweiler and Hinkkanen on semiconjugation of entire functions \cite[Theorem 3]{MR1684251} can readily be extended to the class $\mathcal{M}$; see Section~\ref{pf:count} for details. In particular, this yields the following analogue of Theorem~B. 

\begin{theorem}
\label{theo:count}
If a function $ f \in \mathcal{M} $ is not a M\"{o}bius map, then the set of all elements of $\mathcal{M}$ that permute with $ f $ is countably infinite.
\end{theorem}

If $f$ is an entire function that is not a polynomial of degree less than two, then the iterates of $f$ form a countably infinite set of entire functions that permute with $f$. An analogous remark holds if $f \in\mathcal{M}$. However, if $f$ is a \tmf\ that is not entire then, in general, the iterates of $f$ are not meromorphic in $ \C $.  Indeed, suppose we define a function meromorphic in $ \C $ to be \emph{minimally permuting} if, among such functions, it permutes only with itself and with the identity map. Then we have the following.

\begin{theorem}
\label{theo:onlytwo}
There exist \tmf s that are minimally permuting.
\end{theorem}

The organisation of this paper is as follows. We give the proofs of Theorems \ref{theo:perm} and \ref{theo:permK} in Section \ref{pf:perm}, and the proof of Theorem \ref{pf:count} in Section \ref{pf:count}. Then, in Section \ref{pf:onlytwo}, we prove Theorem~\ref{theo:onlytwo} by giving two examples of minimally permuting \tmf s.
%
%
%
%
%
\section{Proofs of Theorems~\ref{theo:perm} and \ref{theo:permK}}  
\label{pf:perm}

We first show that Theorem \ref{theo:perm} follows easily from results of Goldstein \cite{MR0265595}.

\begin{proof}[Proof of Theorem \ref{theo:perm}]
If $ g $ is an affine map of the form (\ref{defg}), then it follows from Theorem A that there is a \tef\ that permutes with $ g $. Thus it suffices to prove the `only if' direction of Theorem \ref{theo:perm}.

Suppose, then, that $ g $ is a non-constant rational map and that $ f $ is a \tmf\ that permutes with $ g $.  Since there exists $\zeta \in \widehat{\C}$ such that $g(\zeta) = \infty$, it follows that $f(g(\zeta))$ is undefined, so $g(f(\zeta))$ is also undefined by (\ref{eq:perm}). Since $g$ is rational, we deduce that $\zeta = \infty$. Thus infinity is the only pole of $g$, which is therefore a polynomial. It follows by \cite[Theorem~2]{MR0265595} that $g$ is an affine map. The result then follows by \cite[Theorem~11]{MR0265595}.
\end{proof}

We now give the proof of Theorem \ref{theo:permK}, our main generalisation of Theorem~A.  The proof uses certain ideas from iteration theory.  We denote the \emph{iterates} of the function $ f $ by $ f^n:= \underbrace{f \circ f \circ ... \circ f}_{n \text{ times}} $, for $ n \in \N $.  The \emph{Julia set} of a rational map $ g $ of degree at least $ 2 $ is defined to be the set of points in $ \widehat{\C} $ with no neighbourhood in which the iterates of $ g $ form a normal family.  We refer to  \cite{beardon, MR1216719, MR2193309}, for example, for the properties of this set and an introduction to complex dynamics.

\begin{proof}[Proof of Theorem \ref{theo:permK}]
If $ g $ is a M\"{o}bius map that is either parabolic or conjugate to a rational rotation, then it follows from Theorem~A that there is an element of $\mathcal{M}$, conjugate to a {\tef}, that permutes with $ g $. 

If $ g $ is a M\"{o}bius map that is neither parabolic nor conjugate to a rotation, then it is conjugate to a map of the form $z \mapsto \lambda z$, for some $\lambda \in \C$ with $ |\lambda| \neq 0,1 $. Without loss of generality, we can assume that $ g $ is of this form.  We construct a function $ f \in \mathcal{M} $ that permutes with $ g $ as follows. Replacing $\lambda$ with $1/\lambda$ if necessary, we assume that $|\lambda| > 1$. Let $h(z)$ be the function $h(z) := z^2(1-z)^{-2},$ and let
\begin{equation}
\label{nicefun}
f(z) := \sum_{k\in\Z} \lambda^{-k} h(\lambda^k z).
\end{equation}

We claim that the (double) series in (\ref{nicefun}) defines a function which is meromorphic in $\widehat{\C}\setminus\{0, \infty\}$, but in no larger domain in $\widehat{\C}$. In particular $f \in\mathcal{M}$, but $f$ is not a {\tmf}. Since $f$ permutes with the maps $z \mapsto \lambda z$ and $z \mapsto z/\lambda$, this will complete the first part of the proof.

To prove the claim, suppose that $z \in \widehat{\C}\setminus\{0, \infty\}$. Let $U$ be a neighbourhood of $z$ sufficiently small that $\overline{U} \subset \widehat{\C}\setminus\{0, \infty\}$. Since $|\lambda| > 1$, it can be seen that there exists $k_0 \in \N$ such that, for all $w\in U$, we have
\begin{equation*}
|\lambda^{-k} h(\lambda^k w)| = \left|\frac{\lambda^{k} w^2}{(1 - \lambda^k w)^2}\right| <
\begin{cases}
2|\lambda|^{-k}, &\text{for } k \geq k_0, \\
2|w|^2|\lambda|^k, &\text{for } k \leq -k_0.
\end{cases}
\end{equation*}
It follows by the Weierstrass M--test that $f$ restricted to $U$ can be written as the sum of two analytic functions and a rational function, and so $f$ is meromorphic in $U$. Hence $f$ is indeed meromorphic in $\widehat{\C}\setminus\{0, \infty\}$, since $z$ was arbitrary. Moreover, the poles of $f$ are easily seen to be the points $\lambda^k$, for $k \in \Z$. Since these points accumulate on $\{0, \infty\}$, it follows that $f$ cannot be meromorphic in a neighbourhood of either zero or infinity. This completes the proof of our claim.

It remains to prove the `only if' direction of Theorem \ref{theo:permK}. Suppose that $g$ is any non-constant rational map and that $ f \in \mathcal{M}$ permutes with $ g. $  We consider separately the cases where $S(f)$ has one, two, or more than two points. 

First suppose that $S(f)$ is a singleton. Without loss of generality, we can assume that $S(f) = \{ \infty\}$. Then $f$ is a {\tmf}, so it follows from Theorem~\ref{theo:perm} that $ g $ is an affine map of the form (\ref{defg}), and thus is a M\"{o}bius map that is not conjugate to an irrational rotation.  

Suppose next that $S(f)$ has two elements. If there exists $\zeta \in \widehat{\C}\setminus S(f)$ such that $f(\zeta) \in S(f)$, then $S(f^2)$ has more than two elements, so we can replace $f$ with $f^2$. It follows that we can assume that both elements of $S(f)$ are omitted values, and without loss of generality take $S(f) = \{0, \infty\}.$  Thus $f$ is a \emph{transcendental self-map of the punctured plane}. It was pointed out by R{\aa}dstr{\"o}m \cite{MR0056702} that such maps are necessarily of the form
$$
f(z) := z^k \exp(f_1(z) + f_2(1/z)),
$$
where $k \in \Z$ and $f_1, f_2$ are entire functions. If $|z|$ is large, then $f$ behaves like the {\tef} $F(z) := z^k \exp(f_1(z))$. It is straightforward to show that the techniques of \cite{MR0097532, MR0105574} can also be applied in this situation with the same result; we omit the detail.

Finally, consider the case where $S(f)$ has at least three elements. We first show that any rational map $g$ that permutes with $ f $ is a M\"{o}bius map. Suppose, by way of contradiction, that the degree of $g$ is at least two. We deduce by Picard's great theorem and \cite[Theorem 4.1.2]{beardon} that there exists $\alpha \in \widehat{\C} \setminus S(f)$ such that $f(\alpha) \in S(f)$ and the set $\bigcup_{k \geq 0} g^{-k}(\alpha)$ is infinite.

Suppose that $k \geq 0$ and that $\zeta \in g^{-k}(\alpha)$. Since $f^2$ and $g^k$ permute, it follows that $g^k(f^2(\zeta))$ is undefined, and so $\zeta \in f^{-1}(S(f)) \cup S(f)$. We deduce that
\begin{equation}
\label{e1}
\bigcup_{k \geq 0} g^{-k}(\alpha) \subset f^{-1}(S(f)) \cup S(f) = S(f^2).
\end{equation}

It follows from (\ref{e1}), and \cite[Theorem 4.2.7]{beardon} that $\overline{S(f^2)}$ contains the Julia set of~$g$. Since \cite[Theorem 4.2.4]{beardon} the Julia set of $g$ is uncountable and $S(f^2)$ is closed and countable, this is a contradiction, so it follows that $g$ is a M\"{o}bius map.

To complete the proof, it is sufficient to show that $g$ is not of the form $z \mapsto e^{i\theta} z$, where $\theta$ is irrational. If this holds, then $e^{i\theta} f(z) = f(e^{i\theta} z)$, for $z \in \widehat{\C}\setminus S(f)$. Differentiating, we obtain a non-constant function $H \in \mathcal{M}$ with the property that $H(z)~=~H(e^{i\theta} z)$, for $z \in \widehat{\C}\setminus S(f)$. Now choose a point $\xi \in \mathbb{C}\setminus(\{0\}\cup S(f))$. Observe that $$H(e^{ik\theta} \xi) = H(\xi), \qfor k \in \Z.$$ 

Since $ \theta $ is irrational, the points $e^{ik\theta} \xi$, for $k \in\N$, accumulate on the whole circle $\{ w : |w| = |\xi| \}$. This is a contradiction, since all these points are elements of $H^{-1}(H(\xi))$, which can accumulate only on $S(f)$, and $S(f)$ is countable. 
\end{proof}
%
%
%
%
%
\section{Proof of Theorem~\ref{theo:count}}
\label{pf:count}

In \cite[Theorem 3]{MR1684251}, Bergweiler and Hinkkanen gave the following result on semiconjugation of entire functions, which is a generalisation of Theorem B.

\begin{theoremC}
Let $ f $ and $ h $ be entire functions such that $ f $ is not a M\"{o}bius map, and $ h $ is not the identity map.  Then there are only countably many entire functions $ g $ such that 
\begin{equation}
\label{eq:conj}
h \circ g = g \circ f.
\end{equation}
\end{theoremC}

The method of proof of Theorem C can readily be adapted to give the corresponding result for functions in the class $\mathcal{M}$:

\begin{theorem}
\label{theo:BHgen}
Let $ f , h \in \mathcal{M}$ be such that $ f $ is not a M\"{o}bius map, and $ h $ is not the identity map.  Then there are only countably many $ g \in \mathcal{M}$ such that (\ref{eq:conj}) holds.
\end{theorem}

Bergweiler and Hinkkanen's proof of Theorem C uses the facts that, if $ n \in \N $ and~$ f $ is entire but is not a M\"{o}bius map, then $ f^n $ is also entire, and the repelling periodic points of $ f $ are dense in the Julia set $ J(f) $, which is a non-empty perfect set. Here a point $ \zeta \in \widehat{\C} $ is called \emph{periodic} if there exists $p \in \N$ such that $f^p(\zeta) = \zeta$, and it is also called \emph{repelling} if $|f^p(\zeta)| > 1$.

It was shown in \cite{MR1423240} that, if $ f \in \mathcal{M} $ is not a M\"{o}bius map, then the Julia set $ J(f) $ is a non-empty perfect set and that the repelling periodic points of $ f $ are dense in $ J(f) $.  In this case $ J(f) $ is the set of points in $ \widehat{\C} $ at which either some iterate of $ f $ is not defined, or the iterates $ \{ f^n: n \in \N \} $ are all defined but do not form a normal family. 

These results on the properties of the Julia set for functions in $\mathcal{M}$ are all that is needed to adapt the proof of Theorem C and so prove Theorem \ref{theo:BHgen}.  Clearly, Theorem \ref{theo:count} then follows immediately. For completeness, we give a brief proof of Theorem \ref{theo:BHgen} using Bergweiler and Hinkkanen's method.

\begin{proof}[Proof of Theorem \ref{theo:BHgen}]
We will define a countable collection of subsets of $\mathcal{M}$, denoted by $(P_{i,j,k})$, for $i,j,k \in\N$, and show that:
\begin{enumerate}[(i)]
\item every non-constant $g \in\mathcal{M}$ that satisfies (\ref{eq:conj}) lies in  $ P_{i,j,k}$ for some $i,j,k~\in~\N$, and \label{P2}
\item for each $i,j,k \in \N$, the set $P_{i,j,k}$ contains at most one element.\label{P1}
\end{enumerate}
Since there are at most countably many constant functions $ g $ satisfying (\ref{eq:conj}) (because if $ g \equiv c $ then $ c $ is a fixed point of $ h $), it is easy to see that Theorem \ref{theo:BHgen} then follows.

To define the sets $P_{i,j,k}$, we consider the indices $i,j,k$ in turn.  

\begin{itemize}
\item For some $ p \in \N $, $ f^p $ has a repelling fixed point, $ \xi $ say.  Thus we can construct a nested sequence of disks $ (D_i)_{i \in \N} $, centred at $ \xi$, in which $ f^p $ is defined and univalent, with the radius of $ D_i $ tending to $ 0 $ as $ i \to \infty $, and such that a univalent branch $ F $ of $ f^{-p} $ maps each $ D_i $ into a relatively compact subset of itself, with $ F^n(z) \to \xi $ uniformly as $ n \to \infty $ for $ z \in D_i $. Since $ \xi \in J(f) $, it follows from the properties of $ J(f) $ described above that, for each $ i \in \N, $ we can choose $ a_i \in D_i \setminus \{ \xi \} $ and $ p_i \geq 1 $ such that $ f^{p_i}(a_i) = a_i$.
\item Now let $(\eta_j)_{j\in\N}$ be an enumeration of the repelling fixed points of $ h^p $.  Then we argue similarly that, for each $j\in\N$, there is a disk $ K_j $ centred at $ \eta_j $ in which $ h^p $ is defined and univalent, and such that a univalent branch $ H_j $ of $ h^{-p} $ defined on $ K_j $ fixes $ \eta_j $ and maps $ K_j $ into a relatively compact subset of itself, with $ H_j^n(z) \to \eta_j $ uniformly as $ n \to \infty $ for $ z \in K_j $.
\item Finally, let $(b_k)_{k\in\N}$ be an enumeration of all the periodic points of $h$. 
\end{itemize}
Note that, for simplicity, we have assumed here that $ h $ has infinitely many repelling fixed points and periodic points, but the argument remains valid in cases where there are only finitely many. 

For each $i,j, k \in \N$, we now define $ P_{i,j, k}$ to be the set of all non-constant $g \in\mathcal{M}$ that satisfy~(\ref{eq:conj}) and are such that $$g(\xi) = \eta_j, \, \quad g(D_i) \subset K_j \quad\text{ and } \quad g(a_i) = b_k \in K_j.$$ %

To see that property (\ref{P2}) holds, suppose that $ g \in \mathcal{M} $ is non-constant. Note that since $\xi=f^p(\xi)$ it follows from (\ref{eq:conj}) that $ g(\xi) \in \C $ is a fixed point of $ h^p $. Moreover, by a calculation, $ g(\xi) $ is a \emph{repelling} fixed point of $ h^p $, so there exists $j\in\N$ such that $ g(\xi) = \eta_j $. By the continuity of $ g $, there also exists $ i \in \N $ sufficiently large that $ g(\overline{D}_i) \subset K_j. $ Now $a_i=f^{p_i}(a_i)$, and by (\ref{eq:conj}) we have $ g(a_i) = g(f^{p_i}(a_i)) = h^{p_i}(g(a_i)), $ so $ g(a_i) \in K_j$ is a fixed point of $ h^{p_i} $; in other words, there exists $k\in\N$ such that $g(a_i) = b_k$. Thus $ g \in P_{i,j, k} $. 

Finally we show that property (\ref{P1}) also holds. Fix $i,j,k \in \N$, and assume that $ g, \tilde{g} \in P_{i,j,k}. $  Define $ a_{i,n} = F^n(a_i) \in D_i$ and $b_{k,n} = H_j^n(b_k) \in K_j$, for $n \in \N\cup\{0\}.$ 
We claim that $ g(a_{i,n}) = b_{k,n} $, for $ n \in \N \cup \{ 0 \}. $ This is certainly true for $ n = 0 $, so suppose it is true for $ 0 \leq n \leq m-1 $, for some $ m \geq 1. $  The point $ z = a_{i,m} $ is the unique solution in $ D_i $ of the equation $ f^p(z) = a_{i,m-1} $, so
\[ b_{k,m-1} = g(a_{i,m-1}) = g(f^p(a_{i,m})). \]
Hence, by (\ref{eq:conj}), we have $ b_{k,m-1} =  h^p(g(a_{i,m})). $  
Also, the point $ w = b_{k,m} $ is the unique solution in $ K_j $ of the equation $ h^p(w) = b_{k,m-1} $. We deduce that $ g(a_{i,m}) = b_{k,m}, $ which proves the claim. 

The same argument can be applied to $ \tilde{g} $, so we have $ g(a_{i,n}) = \tilde{g}(a_{i,n}) $ for all $ n \in \N. $  Since $ g $ and $ \tilde{g} $ are meromorphic, and $ \lim_{n \to \infty} a_{i,n} = \xi $ is finite, we have $ g \equiv \tilde{g} $ by the identity principle.  Thus $ P_{i,j,k} $ indeed contains at most one element.
\end{proof}

%
%
%
%
%
\section{Proof of Theorem~\ref{theo:onlytwo}}
\label{pf:onlytwo}
In this section we prove Theorem~\ref{theo:onlytwo} by giving two examples of {\tmf}s that are minimally permuting. The function in the first example has a simple form and the proof uses only elementary arguments. The second example involves a more complicated function and the proof uses a result from value distribution theory, but is very short. It seems worthwhile to give both examples. 

For the first example, let $f$ be given by $$f(z) := \frac{1}{z} e^z + z.$$ This function has no fixed points, and so permutes with no constant functions. Moreover, by an application of Theorem~\ref{theo:perm} and an elementary calculation, it can be shown that $f$ permutes with no rational maps apart from the identity map. 

Let $ g $ be a {\tmf} that permutes with $f$; we need to show that $ g = f $. Note that zero is the only pole of $f$, and it is of order one. It follows from (\ref{eq:perm}) that zero is the only pole of $g$. Let $m\in\N$ denote the order of this pole of $ g $.

Let $\zeta$ be a zero of $f$, and let its order be $n\in\N$. Then $\zeta$ is a pole of $g \circ f$ of order $mn$. It follows from (\ref{eq:perm}) that $\zeta$ is a pole of $f \circ g$ of order $mn$, and so $\zeta$ is a zero of $g$ of order $mn$. Similarly, let $\zeta'$ be a zero of $g$, and let its order be $n'\in\N$. Then $\zeta'$ is a pole of $f \circ g$ of order $n'$. It follows from (\ref{eq:perm}) that $\zeta'$ is a pole of $g \circ f$ of order $n'$, and so $\zeta'$ is a zero of $f$ of order $n'/m$.

We deduce that $g/(f)^m$ is a meromorphic function in $\C$ with no poles or zeros, and thus there exists an entire function $H$, with no zeros, such that
\begin{equation}
\label{geq}
g(z) = f(z)^m H(z), \qfor z \in \C.
\end{equation}
\indent We now show that $m=1$ and $H(0) = 1$. To achieve this we consider the zeros of $f$ of large modulus. These points are the solutions of $e^z = -z^2.$  Suppose that $z = x + iy$ is such a point. Since $e^x = |e^z| = |z|^2 = x^2 + y^2$, it follows first that $x$ is large and positive, and then that $y$ is close to $\pm e^{x/2}$. Thus $\arg z$ is close to $\pm \pi/2$ and $y$ is close to a large positive or negative even multiple of $\pi.$   We deduce that, for large positive or negative values of $ n $, there are zeros of $f$ close to the points $\zeta_n := 2 \log{(2|n|\pi)} + 2n\pi i$; we label the corresponding zeros $z_n$. It can be~shown~that 
\begin{equation}
\label{ztends}
z_n - \zeta_n \rightarrow 0 \text{ as } |n|\rightarrow\infty.
\end{equation}

Now there is a neighbourhood of infinity, $U$ say, in which $f$ has an inverse branch, $F$ say, that maps $U$ to a neighbourhood of the origin. Since $f(z) = 1/z + 1 + O(|z|)$ as $z\rightarrow 0$, we have that $F(z) = 1/z + 1/z^2 + O(|z|^{-3})$ as $z \rightarrow \infty$. Let $n_0 \in \N$ be sufficiently large that $z_n \in U$, for $|n| \geq n_0$. Then, by (\ref{eq:perm}) and (\ref{geq}), we have 
\begin{equation}
\label{diffe3}
0 = g(z_n) = g(f(F(z_n))) = f(g(F(z_n))) = f(z_n^mH(F(z_n))), \qfor |n|\geq n_0.
\end{equation}
In other words, $z_n^mH(F(z_n))$ is a zero of $f$, for each $ n $ such that $|n| \geq n_0$. Hence for each such $n$ there exists $p_n \in \Z$ such that 
\begin{equation}
\label{zeq}
z_{p_n} = z_n^mH(F(z_n)).
\end{equation}
Now
\begin{equation}
\label{zeros}
H(F(z)) = H(0) + H'(0)/z + O(|z|^{-2}) \text{ as } z\rightarrow\infty.
\end{equation}
We deduce, by (\ref{ztends}) and (\ref{zeros}), that 
\begin{equation}
\label{zetaeq}
\zeta_{p_n} \sim \zeta_n^m H(0) \text{ as } |n| \rightarrow \infty.
\end{equation}

It is easy to see that $|p_n| \rightarrow \infty$ as $|n|\rightarrow\infty$. Since $\arg \zeta_n \rightarrow \pm \pi/2$ as $|n|\rightarrow \infty$, it follows from (\ref{zetaeq}) that $H(0) = \pm i^{1-m}|H(0)|$.   Therefore, as $|n|\rightarrow\infty$,
\begin{align*}
\zeta_{p_n} \sim \zeta_n^m H(0) & = (2n\pi i)^m\left(1 - \frac{i \log (2|n|\pi)}{n\pi}\right)^m H(0) \\
& \sim \pm (2n\pi)^m \left(\frac{m\log (2|n|\pi)}{n\pi} + i\right)|H(0)|.
\end{align*}
Since $\zeta_{p_n}$ also satisfies $\operatorname{Im } \zeta_{p_n} = \pm \exp(\operatorname{Re } \zeta_{p_n}/2)$, it follows first that $m=1$, then that $\operatorname{Im } \zeta_{p_n} $ and $\operatorname{Im } \zeta_n $ have the same sign, and finally that $H(0) = 1$. 

Thus, for $ |n| $ sufficiently large, we have $p_n = n$ and hence $ H(F(z_n)) = 1 $, by (\ref{zeq}).  Since the points $ F(z_n) $ accumulate at the origin, and $H(0) = 1$, it follows by the identity principle that $ H(z) \equiv 1 $, and this completes the proof that $ f $ is minimally permuting. \\

For the second example, let $q \in \N$ be greater than $16$, and let $p_1, p_2, \ldots, p_q$ be the zeros of the polynomial given in  \cite[Theorem~1]{MR1697581} (the same result was proved independently in \cite{MR1700781}). We then set 
$$
h(z) := \frac{e^z}{\prod_{j=1}^q (z-p_j)}  +  z.
$$
Then $h$ has no fixed points, and it can be shown using Theorem~\ref{theo:perm} that $h$ permutes with no rational maps apart from the identity. We denote the set of poles of $h$ by $S := \{p_1, \ldots, p_q\}$.

Suppose that $g$ is a {\tmf} that permutes with $h$. Then $S$ is also the set of poles of $g$ by (\ref{eq:perm}). Suppose that $z \in g^{-1}(S)$. Then (\ref{eq:perm}) also gives that $g(h(z)) = \infty$, so $z \in h^{-1}(S)$, and we deduce that $g^{-1}(S) = h^{-1}(S)$. It then follows from \cite[Theorem~1]{MR1697581} that $h = g$. Thus $h$ is indeed minimally permuting.

\subsection*{Acknowledgment.}
The authors are grateful to the referee for helpful comments.

\end{document}